\documentclass[a4paper,twoside]{article}
\usepackage{a4}
\usepackage{amssymb}
\usepackage{amsmath}
\usepackage[all]{xy}
\usepackage[active]{srcltx}
\usepackage[pagebackref,colorlinks,citecolor=blue,linkcolor=blue,urlcolor=blue]{hyperref}
\usepackage[dvipsnames]{color}
\usepackage{upref}
\allowdisplaybreaks[2] 
\newcount\minutes \newcount\hours
\hours=\time \divide\hours 60 \minutes=\hours
 \multiply\minutes-60
\advance\minutes \time
\newcommand{\klockan}{\the\hours:{\ifnum\minutes<10 0\fi}\the\minutes}
\newcommand{\tid}{\today\ \klockan}
\newcommand{\prtid}{\smash{\raise 10mm \hbox{\LaTeX ed \tid}}}
\renewcommand{\prtid}{}
\makeatletter  \headheight 10pt
\def\sectionmark#1{} 
\def\subsectionmark#1{}
\newcommand{\sectnr}{\ifnum \c@secnumdepth >\z@
                 \thesection.\hskip 1em\relax \fi}
\def\@evenhead{\footnotesize\rm\thepage\hfil\leftmark\hfil\llap{\prtid}}
\def\@oddhead{\footnotesize\rm\rlap{\prtid}\hfil\rightmark\hfil\thepage}
\def\tableofcontents{\section*{Contents} 
 \@starttoc{toc}}
\makeatother
\makeatletter
\def\@biblabel#1{#1.}
\makeatother
\makeatletter
\let\Thebibliography=\thebibliography
\renewcommand{\thebibliography}[1]{\def\@mkboth##1##2{}\Thebibliography{#1}
\addcontentsline{toc}{section}{References}
\frenchspacing 
\setlength{\@topsep}{0pt}
\setlength{\itemsep}{0pt}%
\setlength{\parskip}{0pt plus 2pt}%
} \makeatother
\makeatletter
\def\mdots@{\mathinner.\nonscript\!.%
 \ifx\next,.\else\ifx\next;.\else\ifx\next..\else
 \nonscript\!\mathinner.\fi\fi\fi}
\let\ldots\mdots@
\let\cdots\mdots@
\let\dotso\mdots@
\let\dotsb\mdots@
\let\dotsm\mdots@
\let\dotsc\mdots@
\def\vdots{\vbox{\baselineskip2.8\p@ \lineskiplimit\z@
    \kern6\p@\hbox{.}\hbox{.}\hbox{.}\kern3\p@}}
\def\ddots{\mathinner{\mkern1mu\raise8.6\p@\vbox{\kern7\p@\hbox{.}}%
    \raise5.8\p@\hbox{.}\raise3\p@\hbox{.}\mkern1mu}}
\makeatother
\makeatletter
\let\Enumerate=\enumerate
\renewcommand{\enumerate}{\Enumerate%
\setlength{\itemsep}{0pt}%
\setlength{\parskip}{0pt plus 1pt}%
\renewcommand{\theenumi}{\textup{(\alph{enumi})}}%
\renewcommand{\labelenumi}{\theenumi}%
}
\makeatother
\makeatletter
\def\@seccntformat#1{\csname the#1\endcsname.\quad}
\makeatother
\RequirePackage{ifthen}
\newcommand{\authortitle}[3]{\author{#1}\title{#2}%
   \ifthenelse{\equal{#3}{}}{\markboth{#1}{#2}}{\markboth{#1}{#3}}}
\newcommand{\art}[6]{{\sc #1, \rm #2, \it #3 \bf #4 \rm (#5), \mbox{#6}.}}
\newcommand{\auth}[2]{{#1, #2.}}

\newcommand{\artprep}[3]{{\sc #1, \rm #2, #3.}}
\newcommand{\arttoappear}[3]{{\sc #1, \rm #2, to appear in \it #3}}

\newcommand{\AND}{{\rm and }}
\newcommand{\arXiv}[1]{{\tt \href{https://arxiv.org/abs/#1}{arXiv:#1}}}
\RequirePackage{amsthm}
\newtheoremstyle{descriptive}%
  {\topsep}   
  {\topsep}   
  {\rmfamily} 
  {}          
  {\bfseries} 
  {.}         
  { }         
  {}          
\newtheoremstyle{propositional}%
  {\topsep}   
  {\topsep}   
  {\itshape}  
  {}          
  {\bfseries} 
  {.}         
  { }         
  {}          
\theoremstyle{propositional}
\newtheorem{thm}{Theorem}[section]
\newtheorem{prop}[thm]{Proposition}

\theoremstyle{descriptive}
\newtheorem{deff}[thm]{Definition}
\newtheorem{example}[thm]{Example}
\newtheorem{remark}[thm]{Remark}
\makeatletter
\renewenvironment{proof}[1][\proofname]{\par
  \pushQED{\qed}%
  \normalfont
  \trivlist
  \item[\hskip\labelsep
        \itshape
    #1\@addpunct{.}]\ignorespaces
}{%
  \popQED\endtrivlist\@endpefalse
} \makeatother
\newcommand{\setm}{\setminus}

\newcommand{\Cinfty}{{C_\infty}}

\DeclareMathOperator{\Lip}{Lip}

\newcommand*{\coloneq}{:=}

\newcommand{\bdry}{\partial}
\newcommand{\bdy}{\bdry}
{\catcode`p =12 \catcode`t =12 \gdef\eeaa#1pt{#1}}      
\def\accentadjtext#1{\setbox0\hbox{$#1$}\kern   
                \expandafter\eeaa\the\fontdimen1\textfont1 \ht0 }
\def\accentadjscript#1{\setbox0\hbox{$#1$}\kern 
                \expandafter\eeaa\the\fontdimen1\scriptfont1 \ht0 }
\def\accentadjscriptscript#1{\setbox0\hbox{$#1$}\kern   
                \expandafter\eeaa\the\fontdimen1\scriptscriptfont1 \ht0 }
\def\accentadjtextback#1{\setbox0\hbox{$#1$}\kern       
                -\expandafter\eeaa\the\fontdimen1\textfont1 \ht0 }
\def\accentadjscriptback#1{\setbox0\hbox{$#1$}\kern     
                -\expandafter\eeaa\the\fontdimen1\scriptfont1 \ht0 }
\def\accentadjscriptscriptback#1{\setbox0\hbox{$#1$}\kern 
                -\expandafter\eeaa\the\fontdimen1\scriptscriptfont1 \ht0 }

\newcommand{\de}{\delta}
\newcommand{\eps}{\varepsilon}

\renewcommand{\phi}{\varphi}
\newcommand{\p}{{$p\mspace{1mu}$}}
\newcommand{\R}{\mathbf{R}}

\newcommand{\Zp}{\mathbf{Z}^+}
\newcommand{\eR}{{\overline{\R}}}

\newcommand{\Ninfty}{N^{1,\infty}}
\newcommand{\NinftyP}{N^{1,\infty}(\P)}

\newcommand{\Ga}{\Gamma}

\renewcommand{\P}{\mathcal{P}}

\numberwithin{equation}{section}
\newcommand{\eqv}{\ensuremath{
\mathchoice{\ \Longleftrightarrow \ }{\Leftrightarrow}
                {\Leftrightarrow}{\Leftrightarrow}}}
\newcommand{\imp}{\ensuremath{
\mathchoice{\ \Longrightarrow \ }{\Rightarrow}
                {\Rightarrow}{\Rightarrow}}}

\newcommand{\negimp}{\ensuremath{\textstyle\kern 0em \not\kern 0em\Rightarrow}}
\newcommand{\Longnegrevimp}{{\textstyle\kern 0.5em \not\kern -0.5em\Longleftarrow}}
\newcommand{\Downnegimp}{{\textstyle\kern 0.28em \not\kern -0.28em\Big\Downarrow}} 

\begin{document}

\authortitle{Anders Bj\"orn and Jana Bj\"orn}
{Quasicontinuity of $\Ninfty$ functions \\ and the 
Vitali--Carath\'eodory property \\ on general metric spaces}
{Quasicontinuity of $\Ninfty$ functions and the 
Vitali--Carath\'eodory property}
\author{
Anders Bj\"orn \\
\it\small Department of Mathematics, Link\"oping University, SE-581 83 Link\"oping, Sweden\\
\it \small anders.bjorn@liu.se, ORCID\/\textup{:} 0000-0002-9677-8321
\\
\\
Jana Bj\"orn \\
\it\small Department of Mathematics, Link\"oping University, SE-581 83 Link\"oping, Sweden\\
\it \small jana.bjorn@liu.se, ORCID\/\textup{:} 0000-0002-1238-6751
}

\date{Preliminary version, \today}
\date{}

\maketitle

\noindent{\small
{\bf Abstract.} 
This note is a follow up on our recent paper with L.~Mal\'y 
(to appear in \emph{Rev. Mat. Complut.}).
We provide a simple 
example of a compact metric space $\P$ for which 
$L^\infty(\P)$ has the Vitali--Carath\'eodory property,
the Sobolev {$\Cinfty$-}capacity is an outer capacity, but 
the Newtonian space $N^{1,\infty}(\P)$ contains
functions which are not weakly quasicontinuous.
The novelty here is that the Vitali--Carath\'eodory property is satified.
We also obtain some related results about quasicontinuous functions
in $\Ninfty(\P)$ and
a characterization of when $L^\infty(\P)$ has the Vitali--Carath\'eodory property.
}

\medskip


\medskip



\medskip

\noindent {\small \emph{Funding}: 
A.~B. resp.\ J.~B. were supported by the Swedish Research Council,
  grants 2024-04095 resp.\ 2022-04048.
}

\section{Introduction}

This note is a follow up on our paper~\cite{BBLMaly} (joint with L.~Mal\'y).
We refer to that paper for further discussion of the concepts in this note.
We will provide some
rather simple but illuminating observations
concerning capacity and quasicontinuous functions in the Newtonian Sobolev space $\Ninfty(\P)$.

\medskip

\emph{We assume throughout the note that $\P = (\P, d, \mu)$
is a metric measure space equipped with a
positive complete Borel measure 
$\mu$ such that $0 < \mu(B) < \infty$ for every ball $B \subset \P$.
It follows that $\P$ is separable and Lindel\"of.}

\medskip

The  following statements
for the metric measure space $\P$ were studied in~\cite{BBLMaly}
in the setting of Newtonian spaces based on general 
quasi-Banach function lattices\/\textup{:}

\begin{enumerate}
\renewcommand{\theenumi}{\textup{(\Alph{enumi})}}%
\item \label{b-qcont}
Every $u \in \Ninfty(\P)$ is quasicontinuous.
\item \label{b-outer}
$\Cinfty$ is an outer capacity.
\stepcounter{enumi}
\stepcounter{enumi}
\stepcounter{enumi}
\item \label{b-repr}
Every $u \in \NinftyP$ has a quasicontinuous representative $v$, i.e.\
\[
\Cinfty(\{x: v(x)\ne u(x)\})=0.
\]
\item \label{b-wqcont}
Every $u \in \NinftyP$ is weakly quasicontinuous.
\end{enumerate}

See Section~\ref{sect-def} for the definitions.
Labels \ref{b-qcont}--\ref{b-wqcont} are as in~\cite{BBLMaly},
but we have omitted some statements 
that for $L^\infty$ 
are equivalent to \ref{b-outer}, by~\cite[Theorem~1.4]{BBLMaly}.
Moreover, \cite[Proposition~6.1]{BBLMaly}
(see Proposition~\ref{prop-Linfty-gen}) shows that
every quasicontinuous function in $\Ninfty(\P)$
is in fact continuous.
This is, however, not the case for weakly quasicontinuous functions.

In particular, it was shown in~\cite[Theorems~1.2 and~1.4]{BBLMaly} that 
\ref{b-qcont} implies the other properties, but not vice versa.
The following 
is a special case of results in~\cite{BBLMaly}.

\begin{thm} \label{thm-VC-new-intro}
\textup{(\cite[Theorems~1.4 and~1.5]{BBLMaly})}
If
$\P$ is locally compact and $L^\infty(\P)$ has the Vitali--Carath\'eodory
property, then
$\Cinfty$ is an outer capacity, i.e.\ \ref{b-outer} holds.
\end{thm}

We can now improve upon this result in the following way,
cf.\ Theorem~\ref{thm-qcont-char-Linfty} for 
implications not  requiring 
the Vitali--Carath\'eodory property.

\begin{thm} \label{thm-main}
Assume that 
$L^\infty(\P)$ has the Vitali--Carath\'eodory property.
Then 
$\Cinfty$ is an outer capacity, i.e.\ \ref{b-outer} holds.

Moreover, \ref{b-qcont}\eqv\ref{b-repr}\eqv\ref{b-wqcont},
and there are examples when these properties all fail,  and other
examples when they all hold.
\end{thm}

In particular, we will see in Example~\ref{ex-2^{-n}-VC} that
the quasicontinuity properties \ref{b-qcont}, \ref{b-repr} and~\ref{b-wqcont}
are not implied by the Vitali--Carath\'eodory property.
Thus, in some sense,
Theorem~\ref{thm-VC-new-intro} (and \cite[Theorem~1.5]{BBLMaly})
is  sharp.
Unlike all the counterexamples in~\cite{BBLMaly}, 
this example has the Vitali--Carath\'eodory property.

The Vitali--Carath\'eodory property for $L^\infty(\P)$
can be characterized in the following way.

\begin{prop}  \label{prop-VC}
The Vitali--Carath\'eodory property for $L^\infty(\P)$ holds if and only if
$\mu(\{x\})>0$ for every $x\in\P$.

In particular, $\P$ has to be finite or countable.
\end{prop}

When there are ($L^\infty$-almost) no  nonconstant rectifiable curves
in $\P$,
the following result gives some further equivalent statements
to the Vitali--Carath\'eodory property.
Some statements about (weak) quasicontinuity of functions in $\Ninfty(\P)$
were for such spaces $\P$ given in~\cite[Theorem~6.2]{BBLMaly}.

\begin{prop} \label{prop-VC-no-curve}
Assume that 
there are 
$L^\infty$-almost no nonconstant rectifiable curves
in $\P$.
Then the following are equivalent:
\begin{enumerate}
\item \label{c-VC}
The Vitali--Carath\'eodory property holds for $X=L^\infty$.
\item  \label{c-outer}
$\Cinfty$ is an outer capacity.
\item \label{c-E}
$\Cinfty(E)=1$ for every nonempty set $E$.
\item \label{c-mu}
$\mu(\{x\})>0$ for every $x\in\P$.
\end{enumerate}
\end{prop}

Another property important for quasicontinuity is the 
density of continuous functions in $\Ninfty(\P)$.
This property will, however, not be considered in this note.
It is easily shown that if $\P$ has $L^\infty$-almost no curves, then
continuous functions are dense in $\Ninfty(\P)=L^\infty(\P)$ if
and only if all points in $\P$ are isolated,
which in turn implies that
all functions in $\Ninfty(\P)$ are continuous.

For locally complete connected metric spaces $\P$,
equipped with a doubling measure,
Durand-Cartagena--Jaramillo--Shan\-mu\-ga\-lin\-gam~\cite[Theorem~3.1]{DCJS16}
showed that $\NinftyP=\Lip^\infty(\P)$, with comparable seminorms,
 if and only if $\P$ supports an $\infty$-Poincar\'e inequality, 
which  in turn is equivalent
 to $\P$ being very thick quasiconvex.
See also  Garc\'\i a-Bravo--Ikonen--Zhu~\cite{GB-I-Z} for similar results with infinitesimally doubling measures.
However, Lipschitz continuity is a stronger requirement than 
continuity (=quasicontinuity), as seen for example by the slit disc.
Note that on nontrivial connected spaces, the Vitali--Carath\'eodory property 
fails,  by Proposition~\ref{prop-VC}.

\section{Definitions and some results from \texorpdfstring{\cite{BBLMaly}}{[1]}}
\label{sect-def}

In~\cite{BBLMaly} we considered Newtonian spaces based
on general quasi-Banach function lattices.
In this note, we limit ourselves to Newtonian spaces 
based on   
$L^\infty$, which is a Banach function space.

A Borel function $g: \P \to [0,\infty]$ 
is an \emph{upper gradient} of $u: \P \to \eR:=[-\infty,\infty]$ if
\begin{equation*} 
 |u(\gamma(0)) - u(\gamma(l_\gamma))| \le \int_\gamma g\,ds
\end{equation*}
for every nonconstant rectifiable curve $\gamma: [0, l_\gamma]\to\P$. 
Here  we  use the convention that
$|(\pm\infty)-(\pm\infty)|=\infty$.

The \emph{Newtonian space} $\NinftyP$ is the space
\begin{equation*}  
  \Ninfty(\P)  =  \Bigl\{u\in L^\infty(\P) :
  \|u\|_{\NinftyP}\coloneq \| u \|_{L^\infty(\P)} + \inf_g \|g\|_{L^\infty(\P)} <\infty \Bigr\},
\end{equation*}
where the infimum is taken over all upper gradients $g$ of $u$.
In this note, it is convenient to
assume that functions in $\NinftyP$
 are defined everywhere (with values in $\eR$),
not just up to an equivalence class.
This is essential e.g.\  for the definition of
upper gradients to make sense.

The \emph{\textup{(}Sobolev\/\textup{)} capacity} of a set
$E\subset \P$ is defined as
\[
\Cinfty(E) = \inf\{ \|u\|_{\NinftyP}: u\ge 1 \mbox{ on }E\}.
\]

The capacity $\Cinfty$ is  an \emph{outer capacity} if
\[
    \Cinfty(E) = \inf_{\substack{G\text{ open} \\G \supset E  }} \Cinfty(G)  
    \quad \text{for every set } E \subset \P.
\]
Several equivalent characterizations
of the outer capacity property were given in~\cite[Theorem~1.4]{BBLMaly}.

\begin{deff}
A function $u: \P \to \eR$ is \emph{weakly quasicontinuous}
if for every $\eps>0$ there is a set $E \subset \P$ with
$\Cinfty(E) < \eps$ such that the restriction $u\vert_{\P \setminus E}$ is
continuous.
If the set
$E$ can be chosen open for every $\eps>0$, then $u$ is
\emph{quasicontinuous}.
\end{deff}

$\P$ has \emph{$L^\infty$-almost no nonconstant rectifiable curves} 
if there is a Borel function $\rho \in L^\infty(\P)$
 such that $\int_\gamma \rho\,ds = \infty$ for every 
nonconstant rectifiable curve $\gamma \in \Ga$.

The following result was one of the main results in~\cite{BBLMaly}.

\begin{thm} \label{thm-qcont-char-Linfty} 
\textup{(\cite[Theorem~1.4]{BBLMaly})}
The following implications hold under our general assumptions\/\textup:
\begin{equation*}
\xymatrix{
\ref{b-qcont} \ar@{=>}[r]^{\Longnegrevimp}
\ar@{=>}[d]
& \ref{b-outer} 
\ar@{<=}[d]^{\textstyle\text{\rlap{\kern -0.65em $\not$}} \Downnegimp}
\\
 \ref{b-repr} \ar@{=>}[r]^{\Longnegrevimp} & 
\ref{b-wqcont}\rlap{\textup{.}} 
}
\end{equation*}
Moreover, there are examples when all of these properties fail.
\end{thm}

\begin{prop} \label{prop-Linfty-gen}
\textup{(\cite[Proposition~6.1]{BBLMaly})}
The following are true for any metric measure space $\P$ 
satisfying the standing assumptions\/\textup{:}
\begin{enumerate}
\renewcommand{\theenumi}{\textup{(\roman{enumi})}}%
\item \label{r1}
$\Cinfty(G)=1$ for every nonempty open set $G$.
\item \label{r2}
Every quasicontinuous function is  continuous.
\item \label{r3} 
For each $E \subset \P$ either
$    \Cinfty(E)=0$ or $\Cinfty(E)=1$. 
\end{enumerate}
\end{prop}

\section{The Vitali--Carath\'eodory property}

\begin{deff}   \label{def-VC}
$L^\infty(\P)$ 
has the \emph{Vitali--Carath\'eodory property} 
if
\begin{equation*} 
	\| u \|_{L^\infty(\P)} = \inf \{ \|v\|_{L^\infty(\P)}: 
   v\ge |u| \mbox{  and $v$ is lower semicontinuous on $\P$}\}
\end{equation*}
for every measurable function $u: \P \to \eR$.
\end{deff}

\begin{proof}[Proof of Proposition~\ref{prop-VC}]
If $\mu(\{x\})=0$, then $\|\chi_{\{x\}}\|_{L^\infty(\P)} =0$.
Let $v\ge \chi_{\{x\}}$ be a lower semicontinuous function.
Then $v>\tfrac12$  in an open set $G\ni x$ (which has positive measure) and hence 
$\|v\|_{L^\infty(\P)}> \tfrac12$.
This shows that the Vitali--Carath\'eodory property fails.

Conversely, assume that $\mu(\{x\})>0$ for every $x\in\P$ and $u\in L^\infty(\P)$.
Then, for every $u\in  L^\infty(\P)$, the constant function
\[
v \equiv \|u\|_{L^\infty(\P)} = \sup_{x\in \P} |u(x)|
\]
is a lower semicontinuous 
majorant of $|u|$ with $\|v\|_{L^\infty(\P)} = \|u\|_{L^\infty(\P)}$.
\end{proof}

\begin{proof}[Proof of Proposition~\ref{prop-VC-no-curve}]
As there are 
$L^\infty$-almost no nonconstant rectifiable curves,
$\Cinfty(E)=\|\chi_E\|_{L^\infty(\P)}$.
That \ref{c-outer}--\ref{c-mu} are equivalent
then follows from Proposition~\ref{prop-Linfty-gen}.

The equivalence \ref{c-VC}\eqv\ref{c-mu}
follows directly from Proposition~\ref{prop-VC}.
\end{proof}

\begin{proof}[Proof of Theorem~\ref{thm-main}]
By Proposition~\ref{prop-VC}, $\P$ is at most  countable
and thus contains no nonconstant curves.
It then follows from Proposition~\ref{prop-VC-no-curve}
that $\Cinfty$ is an outer capacity,
and moreover that 
$\Cinfty(E)=1$ for every nonempty set $E$.
Hence any weakly quasicontinuous continuous function is continuous
and thus quasicontinuous, i.e.\ 
\ref{b-wqcont}\imp\ref{b-qcont}.
The implications
\ref{b-qcont}\imp\ref{b-repr}\imp\ref{b-wqcont} follow
from Theorem~\ref{thm-qcont-char-Linfty}.

Example~\ref{ex-2^{-n}-VC} shows that 
\ref{b-qcont}, \ref{b-repr} and~\ref{b-wqcont} can fail.
On the other hand, \ref{b-qcont}, \ref{b-repr} and~\ref{b-wqcont}
trivially hold when $\P$ is a finite or discrete  set.
\end{proof}

\begin{remark} \label{rmk-1.4}
All counterexamples in~\cite{BBLMaly}, showing 
the negated implications, fail the Vitali--Carath\'eodory property.
The new Example~\ref{ex-2^{-n}-VC}
shows that 
\ref{b-outer}\negimp\ref{b-qcont} and
\ref{b-outer}\negimp\ref{b-wqcont} 
for $\NinftyP$
even when the Vitali--Ca\-ra\-th\'eo\-do\-ry property for $L^\infty(\P)$ holds.
\end{remark}

\begin{example} \label{ex-2^{-n}-VC}
(This example is a slight modification of
\cite[Example~6.4]{BBLMaly}, 
but with significantly new conclusions.
The difference is that $\P$ is equipped with 
the measure $\sum_{n=1}^\infty 2^{-n} \de_{2^{-n}}$
in \cite[Example~6.4]{BBLMaly}, and therefore the Vitali--Carath\'eodory property
is not satisfied.)

Let 
\[ 
   \P=\{0,2^{-n}:n \in \Zp\}
\quad \text{and} \quad
\mu=\de_0 + \sum_{n=1}^\infty 2^{-n}
\de_{2^{-n}},
\]
where $\de_x$ is the Dirac measure at $x$.
Proposition~\ref{prop-VC-no-curve}
shows that the Vitali--Carath\'eodory property holds
for $L^\infty(\P)$ and that $\Cinfty$ is an outer capacity,
i.e.\ \ref{b-outer} holds.

Moreover, $\chi_{\{0\}} \in \NinftyP$ is not weakly quasicontinuous
and does not have a quasicontinuous representative,
i.e.\ \ref{b-qcont}, \ref{b-repr} and~\ref{b-wqcont} fail.
The function $\chi_{\{0\}}$ also
shows that continuous functions are not dense in $\Ninfty(\P)$.
\end{example}

\section{Local connectedness}

We recall the classical definition.

\begin{deff} \label{deff-lc}
$\P$ is \emph{locally connected at $x\in \P$}
if for every $r>0$ there is a connected neighbourhood
(not necessarily open) $G \subset B(x,r)$ of $x$.

$\P$ is \emph{locally connected} if it is locally connected at
every point $x \in \P$.
\end{deff}

(Connected) components are always closed.
If $\P$ is locally connected then every 
component is open.
This is never the case when $\P$ is not locally connected.

\begin{prop} \label{prop-not-loc-conn}
Assume that $\P$ is not locally connected.

Then there is a non-continuous function $u \in \NinftyP$,
and moreover $u$ is not even quasicontinuous.
\end{prop}

\begin{proof}
We can find $x_0$ and $r$ such that $B:=B(x_0,r)$ does not
contain any connected (not necessarily open) neighbourhood of $x_0$.
Let $E$ be the component of $B$ containing $x_0$.
Then $E$ is relatively closed in $B$ and thus measurable.
By the choice of $r$, we see that 
$x_0 \in \bdy (B \setm E)$.
Let 
\[
    u(x)=\begin{cases}
             1-d(x,x_0)/r, & \text{if } x \in E, \\
             0, & \text{otherwise}.
      \end{cases}
\]
As there are no curves between $E$ and $B \setm E$,
we see that $g \equiv 1/r$ is an
upper gradient of $u$ in $\P$, and thus $u \in \NinftyP$.
Clearly, $u$ is not continuous at $x_0$.
By Proposition~\ref{prop-Linfty-gen}\ref{r2}, $u$ is not even 
quasicontinuous either.
\end{proof}

\begin{remark}
If $\P$ is disconnected, 
then there are two cases:
\begin{enumerate}
\item
The components of $\P$ are ``well separated''  in the following sense:
If 
$x \in \P$ and $\P_x$ is the component containing  $x$, 
then $x \notin \overline{\P \setm \P_x}$.
In this case,  the components are open 
and (quasi)continuity
for functions in $\NinftyP$ can be studied in each component separately.
\item 
There is
$x \in \P$ such that $x \in \overline{\P \setm \P_x}$, where 
$\P_x$ is the component of $x$.
In this case,
$\P$ is not locally connected at $x$.
Hence, by Proposition~\ref{prop-not-loc-conn},
there is a non-continuous function $u \in \NinftyP$
which is not even quasicontinuous.
\end{enumerate} 
\end{remark}

\end{document}